\documentclass[12pt,final]{article}
\usepackage{amsmath,amsthm,amsfonts,amssymb,graphicx,enumerate,psfrag}
\usepackage{mathrsfs}
\usepackage{fullpage}

\usepackage[colorlinks,citecolor=blue,urlcolor=blue]{hyperref}
\usepackage[utf8]{inputenc} 
\newtheorem{lemma}{Lemma}
\newtheorem{assumption}{Assumption}
\newtheorem{theorem}{Theorem}

\newtheorem{proposition}{Proposition}
\newtheorem{corollary}{Corollary}
\newtheorem{remark}{Remark}

\newcommand{\dN}{\mathbb {N}}
\newcommand{\dZ}{\mathbb {Z}}
\newcommand{\cF}{\mathcal {F}}
\newcommand{\cU}{\mathcal {U}}
\newcommand{\id}{\circ}
\newcommand{\cV}{\mathcal {V}}
\newcommand{\cA}{\mathcal {A}}
\newcommand{\cI}{\mathcal {I}}
\newcommand{\dR}{\mathbb {R}}
\newcommand{\ee}{e}

\newcommand{\EE}{{\mathbb{E}}}
\newcommand{\PP}{{\mathbb{P}}}
\newcommand{\dX}{{\mathbb{X}}}
\newcommand{\diam}{{\mathrm{diam}}}
\newcommand{\tmix}{{\mathrm{t}_{\textsc{mix}}}}
\newcommand{\trel}{{\mathrm{t}_{\textsc{rel}}}}

\newcommand{\Var}{{\mathrm{Var}}}
\newcommand{\Val}{{\mathrm{val}}}
\newcommand{\Varent}{{\mathrm{Varent}}}

\title{Concentration of information on discrete groups}
\author{Jonathan Hermon, Xiangying Huang, Francesco Pedrotti, Justin Salez}
\begin{document}
\maketitle
\abstract{Motivated by the Asymptotic Equipartition Property and its recently discovered role in the cutoff phenomenon, we initiate the systematic study of varentropy on discrete groups. Our main result is an approximate tensorization inequality which asserts that the varentropy of any conjugacy-invariant random walk is, up to a universal multiplicative constant, at most that of the free Abelian random walk with the same jump rates. In particular, it is always bounded by the number $d$ of generators, uniformly in time and in the size of the group. This universal estimate is sharp and can be seen as a discrete analogue of a celebrated result of Bobkov and  Madiman concerning random $d-$dimensional vectors with a log-concave density (AOP 2011). A key ingredient in our proof is the fact that conjugacy-invariant random walks have non-negative Bakry-\'Emery curvature, a result which seems new and of independent interest.}

\tableofcontents
\section{Introduction}
\subsection{The Asymptotic Equipartition Property}
Let $X$ be  a random variable  taking values in a measured space $(\dX,\cA,\pi)$ and admitting a density $f$ w.r.t. $\pi$. In Information Theory, the real-valued random variable
\begin{eqnarray}
\label{def:info}
\cI(X) & := & -\log f(X),
\end{eqnarray}
is known as the \emph{information content} of $X$, because it quantifies the level of ``surprise'' or ``unpredictability'' inherent to the actual observation of $X$; see, e.g.,  \cite{MR2239987}.  Its expectation $H(X):=\EE[\cI(X)]$ is nothing but the  \emph{entropy} of $X$, which plays a fundamental role in many branches of mathematics, physics and computer science.

An essential property of the information content is \emph{tensorization}: if $X_1,\ldots,X_d$ are independent, each taking values in its own measured space, then the collection $X=(X_1,\ldots,X_d)$ (taking values in the resulting product space) satisfies
\begin{eqnarray}
\cI(X) & = & \cI(X_1)+\cdots+\cI(X_d).
\end{eqnarray}
In particular, when $X_1,\ldots,X_d$ are i.i.d. and $d$ is large, $\cI(X)$ is highly concentrated around $dH(X_1)$, by virtue of the law of large numbers: in words, the outcome of an experiment consisting of a huge number of independent and identically distributed observations is very likely contained in a deterministic set of ``typical outcomes'', all of which are ``roughly equally surprising''. This elementary fact happens to be the simplest instance of a very general phenomenon discovered by Shannon \cite{MR4679166}, McMillan \cite{MR55621} and Breiman \cite{MR92710}, and known as the \emph{Asymptotic Equipartition Property}. It applies to any ergodic process and can  be considered as the very starting point of Information Theory. We refer the reader to the classical textbook \cite{MR2239987} for details. 

 As any limit theorem, the Asymptotic Equipartition Property needs to be made quantitative in order to become truly effective, and this naturally amounts to controlling the  variance of the information content:
\begin{eqnarray}
\label{def:varent}
\Varent(X) & := & \Var\left(\cI(X)\right).
\end{eqnarray}
Because it quantifies the fluctuations around the entropy, this fundamental statistics has been beautifully called \emph{varentropy}. It appeared a decade ago in the context of optimal data compression \cite{inproceedings}, but has since then been shown to play an important role in  completely different areas, such as importance sampling \cite{MR3784496} or the cutoff phenomenon \cite{MR4780485,MR4765357}. Just like the information content and entropy, the varentropy tensorizes: When $X_1,\ldots,X_d$ are independent, the collection $X=(X_1,\ldots,X_d)$ satisfies
\begin{eqnarray}
\Varent(X) & = & \Varent(X_1)+\cdots+\Varent(X_d).
\end{eqnarray}
In particular, while the entropy of a collection of i.i.d. random variables grows  linearly with the dimension $d$, the fluctuations of the information content around the entropy are of order $\sqrt{d}$ only: this constitutes a quantitative, non-asymptotic version of the Asymptotic Equipartition Property. Extending the latter to dependent data is a natural and important problem, on which very little progress has been made. 

An emblematic result in this direction is the celebrated varentropy estimate of Bobkov and  Madiman \cite{MR2857249} for log-concave random vectors. Specifically, let $(\dX,\cA,\pi)$ be the $d-$dimensional Euclidean space $\dR^d$ equipped with its Borel $\sigma-$field and the Lebesgue measure, and assume that the $\dX-$valued random variable $X=(X_1,\ldots,X_d)$ admits a density of the form $f=e^{-V}$, where $V$ is convex. Then, Bobkov and  Madiman proved  the existence of a universal constant $c<\infty$ such that
\begin{eqnarray}
\label{logconcave}
\Varent(X) & \le & cd,
\end{eqnarray}
regardless of the distributional details of $X$. In other words, the information content of a log-concave random vector concentrates at least as well as if its coordinates were independent.  This remarkable result has since then been reproved, refined and extended in several directions. In particular, the value of the constant  has been sharpened to $c=1$ \cite[Corollary 6]{MR3132733}, which turned out to be optimal \cite{MR3565259}. As noted by Chewi \cite{MR4651220}, the estimate (\ref{logconcave}) can be seen as a special case of a certain dimensional improvement of the celebrated Brascamp-Lieb inequality \cite[Proposition 4.1]{MR3758731}. Moreover, the log-concavity assumption can be somewhat relaxed, at a well-understood price \cite{MR4073681}. 

For several reasons -- one being the cutoff phenomenon, discussed below -- it would be highly desirable to develop an appropriate analogue of the above result on discrete spaces,  such as the $d-$dimensional Euclidean lattice $\dZ^d$ and its quotients. However, finding the right substitute for log-concavity seems far from obvious, and the various convex-analytical tools and changes of variables used in the proofs of (\ref{logconcave}) do not seem to admit a clear discrete counterpart. The main contribution of the present paper is a universal varentropy estimate of the form (\ref{logconcave}) for conjugacy-invariant random walks on arbitrary discrete groups, with the role of the dimension $d$ being naturally played by the number of generators. This result actually follows from a more refined, time-dependent statement which compares the varentropy of the walk to that of its ``free Abelian version'', in which the $d$ coordinates evolve independently and without torsion. 

\subsection{Setup and main result}
In the remainder of the paper,  $\dX$ is a discrete group, equipped with the counting measure. We write $\id$ for the identity element, $x^{-1}$ for the inverse of $x$, and $xy$ for the product of $x$ by $y$. By definition, the varentropy of  a  $\dX-$valued random variable $X$ is  
\begin{eqnarray}
\Varent(X) & = & \Var(\log f(X)),
\end{eqnarray}
where $f\colon x\mapsto \PP(X=x)$ is the probability mass function of $X$. As motivated in the first section, our aim is to investigate the varentropy of certain \emph{random walks} on $\dX$. 

Specifically, we fix  a finitely-supported probability vector $\mu$ on $\dX$, and we consider the continuous-time  Markov chain $(X_t)_{t\ge 0}$ whose state space is $\dX$, whose initial state is $X_0=\id$, and whose generator acts on bounded functions $f\colon \dX\to\dR$ as follows:
\begin{eqnarray}
\label{def:L}
\forall x\in\dX,\qquad (Lf)(x) & := & \sum_{z\in\dX}\mu(z)\left(f(xz)-f(x)\right).
\end{eqnarray}
In more concrete terms, the current position gets right-multiplied by a random sample from $\mu$ at rate one. Beyond the finiteness of its support and the (harmless) normalization, the only structural assumptions that we shall impose on  $\mu$ are the following:
\begin{assumption}[Reversibility and conjugacy invariance]The measure $\mu$ satisfies:
\begin{enumerate}[(i)]
\item[A1.] Conjugacy invariance: $\mu(xzx^{-1})=\mu(z)$, for all $x,z\in\dX$.
\item[A2.] Reversibility: $\mu(z)=\mu(z^{-1})$, for all $z\in\dX$.
\end{enumerate}
\end{assumption}
Note that the first condition is trivially satisfied when the group $\dX$ is Abelian. An emblematic non-Abelian example is of course the \emph{transposition walk}, where  $\dX=\mathfrak S_n$ is the symmetric group of order $n$ and $\mu$ the uniform measure on the set of transpositions. 

In general, the value of $\mu(\id)$ is irrelevant in the definition (\ref{def:L}), so we may safely assume that $\mu(\id)=0$. We can then write
\begin{eqnarray}
\mu & = & \sum_{i=1}^d \mu_i\left(\delta_{\ee_i}+\delta_{\ee_i^{-1}}\right),
\end{eqnarray}
for some $d\in\dN$, some rates $\mu_1,\ldots\mu_d>0$ and some group elements $\ee_1,\ldots,\ee_d$, henceforth referred to as the \emph{generators} of the random walk. To formulate our main result, we introduce the function $\cV$ defined as follows: 
\begin{eqnarray}
\label{def:V}
\forall t\ge 0,\qquad \cV(t) & := & t\,\log^2\left(1+ \frac{1}{\sqrt{t}}\right).
\end{eqnarray}
This function is easily seen to be continuously increasing from $\cV(0)=0$ to $\cV(+\infty)=1$. 
\begin{theorem}[Main estimate]\label{th:main}There exists a universal constant $c$  such that
\begin{eqnarray*}
\forall t\ge 0,\qquad \Varent(X_t) & \le & c\,\sum_{i=1}^d \cV\left(\mu_i t\right).
\end{eqnarray*}
Moreover, this estimate is sharp in the sense that the converse inequality holds (with a different value of $c$) when $(X_t)_{t\ge 0}$ is the random walk on $\dZ^d$ which moves by $\pm 1$ unit in the $i$-th direction at rate $\mu_i$, independently for each $1\le i\le d$.
\end{theorem}
Since the function $\cV$ is bounded, our result readily leads to a discrete analogue of the Bobkov-Madiman estimate (\ref{logconcave}), with the role of the dimension being naturally played by the number of generators of the walk:
\begin{corollary}[Discrete analogue of (\ref{logconcave})]\label{coro:d}There exists a universal constant $c$ such that
\begin{eqnarray*}
\forall t\ge 0,\qquad\Varent(X_t) & \le & cd.
\end{eqnarray*}
\end{corollary}
\begin{remark}[Explicit values]We will see in Section \ref{sec:free} that Theorem \ref{th:main} holds with $c=43$ while Corollary \ref{coro:d} holds with $c=16$. We have not tried to optimize those values.  
\end{remark}
Here again, the result is sharp except for the value of $c$, since the varentropy of the simple random walk on $\dZ^d$ is asymptotic to $d/2$ in the large-time limit. Beyond this pleasant consequence, Theorem \ref{th:main} has an interesting interpretation: it  asserts that the varentropy of a reversible, conjugacy-invariant random walk is (up to universal multiplicative constants), at most the varentropy of the free Abelian random walk with the same rates. In view of the product structure of the latter, this principle can be seen as an \emph{approximate tensorization} property for varentropy, reminiscent of the one that has been developed for the entropy of various spin systems with weak dependencies (see the recent work \cite{caputo2024entropyfactorizationcurvature} and the references therein). Extending this approximate tensorization of varentropy beyond the present  group-theoretical setup would lead to a substantial progress on our understanding of the cutoff phenomenon. 
\subsection{Relation to the cutoff phenomenon} Discovered four decades ago in the context of card shuffling \cite{aldous1986shuffling,aldous1983mixing,diaconis1996cutoff}, the cutoff phenomenon is an abrupt transition from out of equilibrium to equilibrium undergone by certain Markov processes in the limit where the size of the state space tends to infinity: instead of decaying gradually over time, their distance to equilibrium remains close to the maximal value for a while and suddenly drops to zero as the time parameter reaches a critical threshold. Despite the accumulation of many examples, this phenomenon is still far from being understood, and identifying the general conditions that trigger it has become one of the biggest challenges in the quantitative analysis of finite Markov chains. Very recently, one of the authors provided a sharp criterion for the occurrence of a cutoff \cite{MR4780485,MR4765357}, based on varentropy. In a nutshell, the latter reads
\begin{eqnarray}
\label{criterion}
\frac{\tmix}{\trel} & \gg & 1+\sqrt{\Varent(X_{\tmix})},
\end{eqnarray}
where $\tmix$ and $\trel$ respectively denote the mixing time and relaxation time of the chain, and where the notation $a \gg b$ means that the ratio $a/b$ diverges as the number of states tend to infinity. We refer the unfamiliar reader to \cite{MR4780485} for  details. 
The potential use of this criterion to explain and predict cutoff clearly motivates the development of a varentropy theory for Markov processes on discrete state spaces, and this is exactly the program that we initiate in the present work. 

A first, naive varentropy estimate was actually obtained  in \cite{MR4780485} for non-negatively curved chains, such as our conjugacy invariant random walk: it reads
\begin{eqnarray}
\label{prior}
\forall t\ge \frac{1}{4}\diam(\dX),\qquad \Varent(X_t) & \le & c\, t\log^2\left\{\frac{1}{\mu_{\min}}\right\},
\end{eqnarray}
where $c$ is a universal constant, $\mu_{\min}$ the minimum non-zero entry of $\mu$, and  $\diam(\dX)$ the diameter of the Cayley graph generated by the support of $\mu$.  As demonstrated in \cite{MR4780485}, this is already sufficient to imply cutoff for a number of interesting models, including the simple random walk on `almost-all' Abelian Cayley graphs of logarithmic degree. Our estimate is always better than (\ref{prior}): indeed, for $t\ge \frac{1}{4}\diam(\dX)$, it reads
\begin{eqnarray*}
\Varent(X_t) & \le & ct \,\sum_{z\in\dX}\mu(z)\log^2\left\{1+\sqrt{\frac{1}{\mu(z)t}}\right\}\\
& \le & ct\log^2\left\{1+\sqrt{\frac{1}{\mu_{\min}t}}\right\}\\
& \le & ct\log^2\left\{1+\sqrt{\frac{4}{\mu_{\min}\diam(\dX)}}\right\}\\
& \le & ct\log^2\left\{1+\sqrt{\frac{4}{\mu_{\min}}}\right\}\\
& \le & \widetilde{c}\, t\log^2\left\{\frac{1}{\mu_{\min}}\right\},
\end{eqnarray*}
and each of those inequalities can be arbitrary loose for certain choices of the parameters $(\dX,\mu,t)$.
In addition, our result has a number of substantial  advantages over (\ref{prior}):
\begin{enumerate}
\item It is uniformly bounded in time, instead of diverging as $t\to\infty$.
\item It also applies to infinite groups, whereas (\ref{prior}) is void when $\diam(\dX)=\infty$. 
\item It remains valid at short times, instead of being restricted to macroscopic scales. 
\item It is not affected by unlikely jumps, as opposed ot the factor $\log^2(\frac{1}{\mu_{\min}})$ in   (\ref{prior}). 
\end{enumerate}
 Thus, using our main estimate instead of (\ref{prior}) immediately leads to a much broader class of random walks exhibiting cutoff. To keep the exposition simple, 
we only mention one emblematic application, which does not even use the full strength of Theorem \ref{th:main} but only the simpler  bound appearing in Corollary \ref{coro:d}.
\begin{corollary}[A simple new condition for cutoff]For reversible,  conjugacy-invariant random walks on arbitrary finite groups, the cutoff phenomenon occurs as soon as
\begin{eqnarray*}
 \frac{\tmix}{\trel} & \gg & \sqrt{d},
 \end{eqnarray*} 
where $d$ denotes the number of generators of the walk. 
\end{corollary}
\section*{Acknowledgment}This work is supported by the ERC consolidator grant CUTOFF (101123174). Views and opinions expressed are however those of the authors only and do not necessarily reflect those of the European Union or the European Research Council Executive Agency. Neither the European Union nor the granting authority can be held responsible for them.
\section{Proof of Theorem \ref{th:main}}

The starting point of our proof is the new observation that conjugacy-invariant random walks on groups are non-negatively curved in the sense of Bakry-\'Emery (see Section \ref{sec:curvature}). By virtue of the celebrated \emph{local Poincaré inequality}, this reduces the task of controlling the varentropy to that of estimating the expected squared gradient of the logarithmic heat kernel. To achieve the latter, we introduce and analyze  a particular measure-preserving transformation  on the space of trajectories, which exploits conjugacy invariance in a crucial way (see Section \ref{sec:logheat}). The sharpness of our estimate is finally established in Section \ref{sec:free}, through the explicit analysis of the varentropy of the free Abelian random walk. 
\subsection{Bakry-\'Emery curvature}
\label{sec:curvature}
Introduced four decades ago in the context of diffusions on manifolds \cite{MR889476}, the Bakry-\'Emery theory of curvature  has quickly become one of the most powerful tools in the quantitative study of geometric, probabilistic and functional-analytical properties of Markov semi-groups. We refer the reader to the  textbook \cite{MR3155209} for a comprehensive introduction. To keep the exposition simple, we shall here restrict our attention to discrete state spaces, as considered, e.g., in the seminal papers \cite{MR1168070, MR1665591, MR2644381}. We let $\cA:=L^\infty(\dX)$ denote the algebra of bounded functions on $\dX$. Given a Markov generator $L\colon \cA\to\cA$, we define the  associated \emph{carré du champ operator} $\Gamma\colon \cA^2\to\cA$ as follows:
\begin{eqnarray}
\label{def:Gamma}
\Gamma(f,g) & := & \frac{1}{2}\left[L(fg)-fLg-gLf\right].
\end{eqnarray}
Similarly, the \emph{iterated carré du champ operator} $\Gamma_2\colon \cA^2\to\cA$ is given by
\begin{eqnarray}
\label{def:Gamma2}
\Gamma_2(f,g) & := & \frac{1}{2}\left[L\Gamma(f,g)-\Gamma(f,Lg)-\Gamma(g,Lf)\right].
\end{eqnarray}
When $g=f$, we  simply write $\Gamma(f)$ and $\Gamma_2(f)$. 
By definition, the \emph{Bakry-\'Emery curvature} of the generator $L$ is then simply the largest number $\kappa\in[-\infty,\infty)$ such that the  following functional inequality is satisfied:
\begin{eqnarray}
\label{def:CDK}
\forall f\in\cA,\qquad \Gamma_2(f) & \ge & \kappa\, \Gamma(f).
\end{eqnarray}
Over the past years, a lower bound on the Bakry-\'Emery curvature of Markov generators on discrete state spaces has been shown to imply an array of powerful quantitative estimates on the underlying semi-group  \cite{MR3173151,MR3910593,MR3581303,MR4367431,MR2644381,MR3776357}. It turns out that conjugacy-invariant random walks on groups always have non-negative Bakry-\'Emery curvature. Interestingly, reversibility is  actually not even needed. This new result generalizes a classical one due to Klartag, Kozma, Ralli  and Tetali \cite{MR3492631}, asserting that random walks on  Cayley graphs of Abelian groups have non-negative Bakry-\'Emery curvature. 
\begin{theorem}[Conjugacy-invariant random walks are non-negatively curved]The generator (\ref{def:L}) has non-negative Bakry-\'Emery curvature.
\end{theorem}
\begin{proof}
We make two simplifying observations. The first one is that our generator $L$ commutes with the shift operator $T_x\colon\cA\to\cA$ defined by  $(T_xf)(z):=f(xz)$, for any $x\in\dX$. It  follows that  for all $f,g\in\cA$, we also have $\Gamma(T_xf,T_xg)= T_x\Gamma(f,g)$ and $\Gamma_2(T_xf,T_xg)=T_x\Gamma_2(f,g)$. Consequently, the functional inequality (\ref{def:CDK}) needs only be verified at a single point, say the identity element $\id$. The second observation is that the functions $Lf,\Gamma(f,g)$ and $\Gamma_2(f,g)$ remain unchanged if we add a constant to $f$ or $g$. Therefore, we may restrict the proof of (\ref{def:CDK}) to test functions $f\in\cA$ that satisfy $f(o)=0$. With those simplifications at hand, we easily compute:
\begin{eqnarray*}
2L\Gamma(f)(o)  & = &  \sum_{z,w\in\dX}\mu(z)\mu(w)\left[f^2(zw)-2f(z)f(zw)\right];\\
2\Gamma (f,Lf)(o) & = & \sum_{z,w\in\dX}\mu(z)\mu(w)f(z)\left[f(zw)-f(z)-f(w)\right].
\end{eqnarray*}
Consequently, we see that
\begin{eqnarray}
\label{Gamma2L}
4\Gamma_2(f)(o) & = & \sum_{z,w\in\dX}\mu(z)\mu(w)F(z,w)+2\left(\sum_{z\in\dX}\mu(z)f(z)\right)^2,
\end{eqnarray}
where we have introduced the short-hand
\begin{eqnarray*}
F(z,w) & := & f^2(zw)-4f(z)f(zw)+2f^2(z).
\end{eqnarray*}
It is now time to use our assumption (A1): the function 
$(z,w)\mapsto(w,w^{-1}zw)$ is a bijection on $\dX^2$ which preserves the measure $\mu\otimes\mu$. Thus, we may use it as a change of variables in the first sum at (\ref{Gamma2L}) to replace $F(z,w)$ by $F(w,w^{-1}zw)$, or even by $\frac 12\left(F(z,w)+F(w,w^{-1}zw)\right)$. But by definition of $F$, we have
\begin{eqnarray*}
F(w,w^{-1}zw) & = & f^2(zw)-4f(w)f(zw)+2f^2(w),
\end{eqnarray*}
so that 
\begin{eqnarray*}
\frac 12\left(F(z,w)+F(w,w^{-1}zw)\right)  & = & f^2(zw)-2\left(f(z)+f(w)\right)f(zw)+f^2(w)+f^2(z)\\
& = & \left(f(zw)-f(z)-f(w)\right)^2-2f(z)f(w).
\end{eqnarray*}
Inserting this back into the above computation finally gives
\begin{eqnarray*}
4\Gamma_2(f)(o) & = & \sum_{z,w\in\dX}\mu(z)\mu(w) \left(f(zw)-f(z)-f(w)\right)^2,
\end{eqnarray*}which is indeed non-negative, as desired. 
\end{proof}
\begin{remark}[Positive curvature]\label{rk:kappa}When the rate function $\mu$ is supported on elements of order $2$,   the conclusion can be improved as follows: keeping only the diagonal terms in the last sum, and writing $\mu_{\min}$ for the minimum non-zero value of $\mu$, we have
\begin{eqnarray*}
4\Gamma_2(f)(o) & = & \sum_{z,w\in\dX}\mu(z)\mu(w) \left(f(zw)-f(z)-f(w)\right)^2\\
& \ge & \mu_{\min}\sum_{z\in\dX}\mu(z)\left(f(z^2)-2f(z)\right)^2\\
& = & 4\mu_{\min}\sum_{z\in\dX}\mu(z)f^2(z)\\
& = & 8\mu_{\min}\Gamma(f)(o).
\end{eqnarray*}
because $f(z^2)=f(\id)=0$ by assumption. Thus, a conjugacy-invariant random walk generated by elements of order $2$ has  Bakry-\'Emery curvature $\kappa\ge 2\mu_{\min}$. This estimate is sharp, as can be seen by considering the classical example of simple random walk on the hypercube, or the transposition walk on permutations.
\end{remark}
Our interest in the Bakry-\'Emery curvature arises from the classical fact that it  implies a \emph{local Poincar\'e inequality} for the Markov chain under consideration. Here, the word `local' refers to the fact that the underlying measure is the law of the chain at an arbitrary time $t$, rather than the usual equilibrium measure. This result is typically only stated for a more restrictive class of test functions than the one needed here, so we prefer to include a proof for completeness. As in the introduction, we let $(X_t)_{t\ge 0}$ denote  a random walk on $\dX$ with generator (\ref{def:L}), starting from the identity. 
\begin{corollary}[Local Poincar\'e estimate]Fix a function $f\colon\dX\to\dR$ and a time $t\ge 0$. If $\EE\left[\Gamma f(X_t)\right]<\infty$, then $\EE\left[f^2(X_t)\right]<\infty$ and we have
\begin{eqnarray*}
\Var\left[f\left(X_t\right)\right] & \le & 2t\,\EE\left[\Gamma f(X_t)\right].
\end{eqnarray*}
\end{corollary}
\begin{proof}Let us first make the stronger assumption that $f\colon\dX\to\dR$ is bounded. As usual, we let $(P_t)_{t\ge 0}$ denote the  Markov semi-group generated by $L$ (i.e. $P_t=e^{tL}$).  Our starting point is the following well-known integral formula, which is easily checked by differentiating both sides w.r.t. $t$ (see, e.g. \cite[Problem 2.12.a]{Ramon}): for all $t\ge 0$, 
\begin{eqnarray*}
 P_t(f^2)-\left(P_t f\right)^2 & = & 2\int_0^t P_{t-s}\Gamma\left(P_sf\right){\rm d}s.
\end{eqnarray*}
Now, the functional inequality (\ref{def:CDK}) is well known to imply (in fact, to be equivalent to) the following sub-commutation property: for all $f\in\cA$,
\begin{eqnarray*}
\forall s\ge 0,\qquad \Gamma (P_sf) & \le & e^{-2\kappa s}\Gamma f.
\end{eqnarray*}
Inserting this back into the previous identity, we arrive at
\begin{eqnarray}
\label{integral}
P_t(f^2)-\left(P_t f\right)^2 & \le & \left(\int_0^t 2e^{-2\kappa s} {\rm d}s\right)P_{t}\Gamma f.
\end{eqnarray}
When $\kappa\ge 0$, the integral is at most $2t$, and evaluating the resulting inequality at our initial state $\id$  yields exactly the desired result. Now, in the general case where $f$ is not necessarily bounded, we may replace it with its $[-n,n]$-truncation $f_n:=(f\wedge n)\vee{-n}$ and apply the first part of the proof to obtain
\begin{eqnarray*}
\Var\left[f_n\left(X_t\right)\right] & \le & 2t\,\EE\left[\Gamma f_n(X_t)\right].
\end{eqnarray*}
Introducing an independent copy $Y_t$ of $X_t$, this can be rewritten as follows:
\begin{eqnarray}
\label{intermed}
\frac{1}{2}\EE\left[\left(f_n\left(X_t\right)-f_n(Y_t)\right)^2\right] & \le & 2t\,\EE\left[\Gamma f_n(X_t)\right].
\end{eqnarray}
Finally, observe that the $[-n,n]$-truncation enjoys the following monotonicity property:
\begin{eqnarray*}
\forall x,y\in\dX,\qquad \left[f_n(x)-f_n(y)\right]^2 & \underset{n\to\infty}{\uparrow} & \left[f(x)-f(y)\right]^2.
\end{eqnarray*}
Thus, on both sides of (\ref{intermed}), the random variable inside the expectation is almost-surely non-decreasing in $n$, and we may safely invoke  monotone convergence to conclude that
\begin{eqnarray*}
\frac{1}{2}\EE\left[\left(f\left(X_t\right)-f(Y_t)\right)^2\right] & \le & 2t\,\EE\left[\Gamma f(X_t)\right].
\end{eqnarray*}
But the right-hand side is finite by assumption, hence so is the left-hand side. By the independence of $X_t$ and $Y_t$, this forces $\EE[f^2(X_t)]<\infty$, and the result is proved.
\end{proof}
\begin{remark}[Ollivier curvature]Conjugacy-invariant random walks are well known to be non-negatively curved in a different sense, namely according to the Wasserstein-based definition of Ollivier \cite{MR2371483,MR2484937}. This implies the weaker local Poincaré inequality
\begin{eqnarray*}
\forall t\ge 0,\qquad \Var\left[f\left(X_t\right)\right] & \le & t\,\mathrm{Lip}^2(f),
\end{eqnarray*}
where $\mathrm{Lip}(f):=\sup\{|f(xz)-f(x)|\colon x,z\in\dX,\mu(z)>0\}$. The substantial gain obtained by replacing the ``worst-case'' gradient bound appearing on the right-hand side by its ``expected'' version $\EE\left[\Gamma f(X_t)\right]$   will turn out to be crucial for our purpose.
\end{remark}
\begin{remark}[Positive curvature]When the Bakry-\'Emery curvature $\kappa$ is strictly positive, we may alternatively bound the integral in (\ref{integral}) by $1/\kappa$ instead of $2t$ to obtain the uniform-in-time local Poincaré inequality
\begin{eqnarray*}
\forall t\ge 0,\qquad \Var\left[f\left(X_t\right)\right] & \le & \frac{1}{\kappa}\,\EE\left[\Gamma f(X_t)\right].
\end{eqnarray*}
This applies, in particular, when $\mu$ is supported on elements of order $2$ as per Remark \ref{rk:kappa}. It leads to a substantial improvement of our varentropy estimate when $t$ is large. 
\end{remark}
\subsection{Gradient of the logarithmic heat kernel}
\label{sec:logheat}
Fix $t\ge 0$ and let  $f_t\colon x\mapsto\PP(X_t=x)$ denote the probability mass generating function of $X_t$, also known as the \emph{heat kernel}. In view of the above corollary, we have
\begin{eqnarray}
\label{reduction}
\Varent(X_t) & \le & 2t\,\EE\left[\Gamma\log f_t(X_t)\right].
\end{eqnarray}
Thus, our task now boils down to finding a sharp estimate on the  expected squared gradient of the logarithmic heat kernel, and this is exactly the content of Proposition \ref{pr:log-grad} below. We introduce the universal function $\cU\colon\dR_+\to\dR_+$ defined as follows:
\begin{eqnarray}
\label{def:U}
\cU(t) & := & 2t\,\EE\left[\log_+^2\left\{\frac{1+N_{t}}{t}\right\}\right],
\end{eqnarray}
where $\log_+(u):=\max\{\log u,0\}$ is the positive part of the $\log$, and where $N_t$ denotes a Poisson random variable with mean $t$. In the next section, this function will be bounded from above by a constant multiple of the  explicit function $\cV$ defined at (\ref{def:V}). Thus, our main estimate is a consequence of (\ref{reduction}) and the following result. 
\begin{proposition}[Gradient estimate for the log-heat kernel]\label{pr:log-grad}For all $t\ge 0$,
\begin{eqnarray*}
2t\,\EE\left[\Gamma\log f_t(X_t)\right]
& \le & \sum_{z\in\dX}\cU\left(\mu(z)t\right).
\end{eqnarray*}
\end{proposition}
The remainder of this section is devoted to the proof of this proposition. 
In order to estimate the term $\log \frac{f_t(x)}{f_t(xz)}$ that appears when explicitating the definition of $\Gamma\log f_t(x)$, we establish a correspondence between random-walk trajectories that end up at  $x$ and random-walk trajectories  that end up at  $xz$, and use conjugacy-invariance to control the incurred change of measure. This is the content of the following technical lemma, which requires a bit of notation. Given $z\in\dX$ and a sequence $w=(w_1,\ldots,w_n)\in\dX^n$, we introduce the key quantity
\begin{eqnarray}
\label{def:lt}
\ell_z(w)& := & \sum_{j=1}^n{\bf 1}_{w_j=(w_{j+1}\cdots w_n) z (w_{j+1}\cdots w_n)^{-1}}.
\end{eqnarray}
Also, we let $\Val(w):=w_1\cdots w_n$ denote the evaluation of the word $w$ in the group. 
\begin{lemma}[Trajectorial correspondence]\label{lm:phi}Fix $z\in\dX$ and $n\in\dN$. There exists a function $\phi\colon\dX^n\times[n+1]\to\dX^{n+1}$ such that for all $(w,i)\in\dX^{n}\times[n+1]$,
\begin{enumerate}[(i)]
\item $\Val(\phi(w,i)) = \Val(w)z$;
\item $\mu^{\otimes{(n+1)}}(\phi(w,i))=\mu^{\otimes n}(w)\mu(z)$;
\item $|\phi^{-1}(\phi(w,i))| = 1+\ell_z(w)$.
\item A sequence $\eta\in\dX^{n+1}$ is in the image of $\phi$ if and only if $\ell_z(\eta)\ne 0$.
\end{enumerate}
\end{lemma}
\begin{proof}
Let $\phi\colon\dX^n\times[n+1]\to\dX^{n+1}$ be defined as follows: for any $(w,i)\in\dX^{n}\times[n+1]$,
\begin{eqnarray}
\label{def:phi}
\phi(w,i)& := & \left(w_1,\ldots,w_{i-1},\xi,w_i,\ldots,w_n\right),\\
\label{def:xi}
\textrm{ with } \xi & := & (w_iw_{i+1}\cdots w_n)z(w_iw_{i+1}\cdots w_n)^{-1}.
\end{eqnarray}
In  words, the sequence $\phi(w,i)$ is obtained from the sequence $w=(w_1,\ldots,w_n)$ by inserting a certain element $\xi\in\dX$  at the $i-$th position,  $\xi$ being  chosen precisely so that Property (i) is satisfied. The second property is clear, since we have
 \begin{eqnarray*}
 \mu^{\otimes(n+1)}(\phi(w,i)) & = & \mu^{\otimes n}(w)\mu(\xi) \ = \ \mu^{\otimes n}(w)\mu(z),
 \end{eqnarray*}
thanks to Assumption A1. Finally, note that by construction, the pre-image $\phi^{-1}(\eta)$ of a given sequence $\eta=(\eta_1,\ldots,\eta_{n+1})\in\dX^{n+1}$ is exactly the set of pairs  of the form 
$
 \left((\eta_1,\ldots,\eta_{j-1},\eta_{j+1},\ldots,\eta_{n+1}),j\right),
$
where $j\in[n+1]$  satisfies the constraint 
\begin{eqnarray}
\label{constraint:j}
\eta_j & = & (\eta_{j+1}\cdots \eta_{n+1}) z (\eta_{j+1}\cdots \eta_{n+1})^{-1}.
\end{eqnarray}This readily implies (iv). Moreover, when we specialize this to the sequence $\eta:=\phi(w,i)$ defined at (\ref{def:phi}), the  constraint (\ref{constraint:j}) rewrites as follows:
 \begin{itemize}
 \item either $j=i$;
 \item or $j>i$ and $w_{j-1}=(w_j\cdots w_n) z (w_j\cdots w_n)^{-1}$;
  \item or $j<i$ and $w_j=(w_{j+1}\cdots w_{i-1}\xi w_i\cdots w_n) z (w_{j+1}\cdots w_{i-1}\xi w_i\cdots w_n)^{-1}$.
\end{itemize}
But thanks to the definition of $\xi$ at (\ref{def:xi}), the condition in the case $j<i$ simplifies to $w_{j}=(w_{j+1}\cdots w_n) z (w_{j+1}\cdots w_n)^{-1}$, and (iii) follows.
\end{proof}
Let us now exploit this trajectorial correspondence to  relate the probabilities of the two events $\{X_t=x\}$ and $\{X_t=xz\}$. We  use the random walk representation
\begin{eqnarray}
X_t & := & W_1W_2\cdots W_{N_t},
\end{eqnarray}
where  $(W_n)_{n\ge 1}$ denotes a sequence of i.i.d. samples from $\mu$, and $(N_t)_{t\ge 0}$ a unit-rate Poisson point process,  independent of $(W_n)_{n\ge 1}$. With this data at hand, we define 
\begin{eqnarray}
N_{t,z} & :=& \ell_z(W_1,\ldots,W_{N_t}),
\end{eqnarray}
and will later show that this is a Poisson random variable with mean $t\mu(z)$ (Lemma \ref{lm:poisson}). The following identity  relates the events $\{X_t=x\}$ and $\{X_t=xz\}$, as promised. 
\begin{lemma}[Change-of-measure formula]\label{lm:identity}For any $t> 0$ and any $x,z\in\dX$, we have
\begin{eqnarray*}
\PP(X_t=xz,N_{t,z}\ne 0) & = & t\mu(z)\EE\left[\frac{1}{1+N_{t,z}}{\bf 1}_{(X_t=x)}\right].
\end{eqnarray*}
\end{lemma}
\begin{proof}Since $N_{t,z}\le N_t$, we can safely write 
\begin{eqnarray*}
\PP(X_t=xz,N_{t,z}\ne 0)
& = & \sum_{n=0}^{\infty}\PP(X_t=xz,N_{t,z}\ne 0,N_t=n+1)\\
& =&  \sum_{n=0}^{\infty}\PP(N_t=n+1)\PP(W_1\cdots W_{n+1}=xz,\ell_z(W_1,\ldots,W_{n+1})\ne 0).
\end{eqnarray*}
Now, for each $n\ge 0$, we may invoke Lemma \ref{lm:phi} to write
\begin{eqnarray*}
\PP(W_1\cdots W_{n+1}=xz,\ell_z(W_1,\ldots,W_{n+1})\ne 0)  
& = & \sum_{\eta\in \phi(\dX^n\times[n+1])}\mu^{\otimes {(n+1)}}(\eta){\bf 1}_{\Val(\eta)=xz}\\& = & \sum_{(w,i)\in \dX^n\times[n+1]}\frac{\mu^{\otimes {(n+1)}}(\phi(w,i))}{|\phi^{-1}(\phi(w,i))|}{\bf 1}_{\Val(\phi(w,i))=xz}\\
& = & (n+1)\mu(z)\sum_{w\in \dX^n}\frac{\mu^{\otimes {(n)}}(w)}{1+\ell_z(w)}{\bf 1}_{\Val(w)=x}\\
& = & (n+1)\mu(z)\EE\left[\frac{1}{1+\ell_z(W_1,\ldots,W_n)}{\bf 1}_{W_1\cdots W_n=x}\right].
\end{eqnarray*}
Inserting this back into the previous computation and using the Poisson identity $(n+1)\PP(N_t=n+1)=t\PP(N_t=n)$, we arrive at
\begin{eqnarray*}
\PP(X_t=xz,N_{t,z}\ne 0) & = & t\mu(z)\sum_{n=0}^\infty\PP(N_t=n)\EE\left[\frac{1}{1+\ell_z(W_1,\ldots,W_n)}{\bf 1}_{W_1\cdots W_n=x}\right]\\
 & = & t\mu(z)\EE\left[\frac{1}{1+N_{t,z}}{\bf 1}_{X_t=x}\right],
\end{eqnarray*}
which concludes the proof of the identity. 
\end{proof}
As announced earlier, the key quantity $N_{t,z}$ featuring in the above identity happens to have a remarkable distribution, thanks to conjugacy invariance again. 
\begin{lemma}[Poisson law]\label{lm:poisson}
$N_{t,z}$ is a Poisson random variable with mean $t\mu(z)$. 
\end{lemma}
\begin{proof}
Fix $n\in\dN$ and, for each $j\in[n]$, consider the event 
\begin{eqnarray*}
A_j & := & \left\{W_j=(W_{j+1}\cdots W_n)z(W_{j+1}\cdots W_n)^{-1}\right\}.
\end{eqnarray*} 
Clearly, $A_j$ belongs to the $\sigma-$field $\cF_j:=\sigma(W_j,\ldots W_n)$, and moreover
\begin{eqnarray}
\PP\left(A_j|\cF_{j+1}\right) & = & \mu\left(W_{j+1}\cdots W_nzW_n^{-1}\cdots W_{j+1}^{-1}\right) \ = \ \mu(z),
\end{eqnarray}
by conjugacy invariance. This shows that the events $A_1,\ldots,A_n$ are independent, each having probability $\mu(z)$. Consequently, the random variable
\begin{eqnarray*}
\ell_z(W_1,\ldots,W_n) & = & \sum_{j=1}^n{\bf 1}_{A_j},
\end{eqnarray*}
has a Binomial distribution with parameters $n$ and $\mu(z)$. Since $N_t$ is independent of $(W_n)_{n\ge 1}$, we deduce that the conditional law of $N_{t,z}:=\ell_z(W_1,\ldots,W_{N_t})$ given $N_t$ is a Binomial distribution with parameters $N_t$ and $\mu(z)$. But $N_t$ itself is a Poisson random variable with mean $t$, so the result follows from the Poisson thinning theorem. 
\end{proof}
We can finally prove our main gradient estimate. 
\begin{proof}[Proof of Proposition \ref{pr:log-grad}]
By definition, we have
\begin{eqnarray*}
\EE\left[\Gamma\log f_t(X_t)\right]
& = & \frac{1}{2}\sum_{x,z\in\dX}\mu(z)f_t(x)\log^2\left\{\frac{f_t(x)}{f_t(xz)}\right\}\\
& \le & \frac{1}{2}\sum_{x,z\in\dX}\mu(z)\left(f_t(x)\vee f_t(xz)\right)\log^2\left\{\frac{f_t(x)}{f_t(xz)}\right\}.
\end{eqnarray*}
But, by reversibility (Assumption A2), the bijective change of variables $(x,z)\mapsto (xz,z^{-1})$
leaves the summand unchanged, while interchanging the roles of $x$ and $xz$. We may exploit this symmetry to restrict the entire sum to those pairs $(x,z)\in\dX^2$ satisfying $f_t(x)\ge f_t(xz)$, and multiply the result by $2$. This leads precisely to
\begin{eqnarray*}
\EE\left[\Gamma\log f_t(X_t)\right]
& \le & \sum_{x,z\in\dX}\mu(z)f_t(x)\log^2_+\left\{\frac{f_t(x)}{f_t(xz)}\right\}. 
\end{eqnarray*}
Now, fix $x,z\in\dX$. Since $\PP(X_t=xz,N_{t,z}\ne 0)\le \PP(X_t=xz)= f_t(xz)$, Lemma \ref{lm:identity} provides the following estimate:
\begin{eqnarray*}
\frac{f_t(xz)}{f_t(x)} & \ge & \EE\left[\left.\frac{t\mu(z)}{1+N_{t,z}}\right|{X_t=x}\right].
\end{eqnarray*} 
But the function $u\mapsto \log_+^2\left\{\frac{1}{u}\right\}$ is decreasing and convex, so Jensen's inequality yields
\begin{eqnarray*}
\log^2_+\left\{\frac{f_t(x)}{f_t(xz)}\right\} & \le & \EE\left[\left.\log^2_+\left\{\frac{1+N_{t,z}}{t\mu(z)}\right\}\right|X_t=x\right].
\end{eqnarray*}
Inserting this back into the previous computation concludes the proof, since $N_{t,z}$ has the same law as $N_{t\mu(z)}$ by Lemma \ref{lm:poisson}. 
\end{proof}

\subsection{Explicit estimates}
\label{sec:free}
In view of (\ref{reduction}) and Proposition \ref{pr:log-grad}, only two tasks remain in order to complete the proof of our main theorem.
\begin{enumerate}
\item Proving that the function $\cU$ defined at (\ref{def:U}) is bounded by a constant multiple of the explicit function $\cV$. This is the content of Lemma \ref{lm:UV} below.
\item Proving that conversely, $\cV$ is at most a constant multiple of the  varentropy in the special case where $\dX=\dZ$ and $\mu=\frac 12{\bf 1}_{\{-1,+1\}}$. This is Lemma \ref{lm:sharp} below.
\end{enumerate}
\begin{lemma}[Explicit estimate on $\cU$]\label{lm:UV} We have 
\begin{eqnarray*}
\cU(t) & \le & 
\left\{
\begin{array}{ll}
2+\frac{2}{t} & \textrm{ for all }t\ge 0;\\
2t\log^2\left(1+\frac{1}{t}\right) &  \textrm{ for all }t\in [0,e^{-1}].
\end{array}
\right.
\end{eqnarray*}
In particular, we have the uniform bounds $\cU\le 8$ and $\cU\le 21.5\,\cV$.
\end{lemma}
\begin{proof}
Using the classical bound $\log u\le u-1$ in the definition of $\cU$ yields
\begin{eqnarray*}
\cU(t) &  \le & \frac{2}{t}\EE\left[\left(1+N_t-t\right)_+^2\right]\\
& \le & \frac{2}{t}\EE\left[\left(1+N_t-t\right)^2\right]\\
& = & 2+\frac{2}{t}\\
& \le & \frac{2+2e}{\cV(\frac 1e)}\cV(t)
\end{eqnarray*} 
where the last inequality  is valid only when $t\ge \frac{1}{e}$, and uses the fact that $\cV$ is increasing. For $t\le \frac{1}{e}$, we instead use the fact that $\log^2_+$ is concave on $[e,\infty)$ to write
\begin{eqnarray*}
\cU(t) &  \le & 2t\log^2_+\EE\left[\frac{1+N_t}{t}\right] \\
 & = & 2t\log^2\left(1+\frac{1}{t}\right)\\
 & \le & 8\cV(t),
\end{eqnarray*} 
where the last line uses $1+\frac{1}{t}\le \left(1+\frac{1}{\sqrt{t}}\right)^2$. Combining those two estimates, we see that $\cU\le \max\left\{2+2e,8\right\}=8$, and that $\cU\le c\,\cV$, where $c=\max\left\{\frac{2+2e}{\cV(1/e)},8\right\}\approx 21.305$.
\end{proof}
\begin{lemma}[Sharpness]\label{lm:sharp}In the special case where $\dX=\dZ$ and $\mu=\frac 12{\bf 1}_{\{-1,+1\}}$, we have
\begin{eqnarray*}
\forall t\ge 0,\qquad  \Varent(X_t)  & \ge & c\cV(t),
\end{eqnarray*}
for some universal constant $c>0$.
\end{lemma}
\begin{proof}
The functions $t\mapsto \Varent(X_t)$ and $t\mapsto\cV(t)$ are both continuous and  positive on $(0,\infty)$, so their ratio is bounded away from $0$ on every compact subset of $(0,\infty)$.  Thus, we only need to investigate the asymptotics as $t\to 0$ and as $t\to\infty$. To this end, we introduce an independent copy $(\widetilde{X}_t)_{t\ge 0}$ of $(X_t)_{t\ge 0}$ and write 
\begin{eqnarray*}
\Varent(X_t) & = & \EE\left[\log^2_+\left\{\frac{f_t(X_t)}{f_t(\widetilde{X}_t)}\right\}\right]\\
& \ge & \EE\left[\log^2_+\left\{\frac{f_t(X_t)}{f_t(\widetilde{X}_t)}\right\}{\bf 1}_{X_t=0,\widetilde{X}_t\in\{-1,1\}}\right]\\
& = & 2f_t(1)f_t(0)\log^2_+\frac{f_t(0)}{f_t(1)}.
\end{eqnarray*}
Now as $t\to 0$, we have $f_t(0)=1-o(1)$ and $f_t(1)=t/2+o(t)$, so we obtain
\begin{eqnarray*}
\Varent(X_t) & \ge & (1-o(1))t\log^2\left(\frac 1t\right),
\end{eqnarray*}
which is of the same order of magnitude as $\cV(t)$. On the other hand, in the limit $t\to \infty$, the Local Central Limit Theorem easily implies the convergence
\begin{eqnarray*}
\sqrt{t}f_t(X_t) & \xrightarrow[t\to\infty]{d} & g(B),
\end{eqnarray*}
where $B$ is a standard Gaussian random variable and $g$ its density. Since $\widetilde{X}_t$ is an independent copy of $X_t$, we actually have the joint convergence
\begin{eqnarray*}
\left(\sqrt{t}f_t(X_t),\sqrt{t}f_t(\widetilde{X}_t)\right) & \xrightarrow[t\to\infty]{d} & \left(g(B),g(\widetilde{B})\right),\end{eqnarray*}
with $\widetilde{B}$ an independent copy of $B$. 
But $(u,v)\mapsto \log_+^2\frac{u}{v}$ is continuous on $(0,\infty)^2$, so
\begin{eqnarray*}
\log_+^2\left\{\frac{f_t(X_t)}{f_t(\widetilde{X}_t)}\right\} & \xrightarrow[t\to\infty]{d} & \log_+^2\left\{\frac{g(B)}{g(\widetilde{B})}\right\},
\end{eqnarray*}
Taking expectations and invoking Fatou's Lemma, we conclude that
\begin{eqnarray*}
\liminf_{t\to \infty}\Varent(X_t) & \ge & \Varent(B).
\end{eqnarray*}
The right-hand side is well known and easily seen to be equal to $1/2$. Since $\cV(t)\to 1$  as $t\to\infty$, the ratio $\Varent(X_t)/\cV(t)$ does indeed remain bounded away from $0$ as $t\to\infty$, and the proof is complete.
\end{proof}

\bibliographystyle{plain}
\bibliography{draft}
\end{document}